\documentclass[12pt]{amsart}

\usepackage{graphicx,color}
\usepackage{setspace}
\usepackage{amsmath, amssymb,amsthm, amsfonts}

\usepackage{accents}

\newtheorem{thm}{Theorem}[section]
\newtheorem{lma}{Lemma}[section]
\newtheorem{prop}{Proposition}[section]
\newtheorem{cor}{Corollary}[section]

\theoremstyle{definition}

\theoremstyle{remark}
\newtheorem{remark}{Remark}[section]

\numberwithin{equation}{section}

\newcounter{mnotecount}[section]

\newcommand{\R}{\mathbb R}

\newcommand{\be}{\begin{equation}}
\newcommand{\ee}{\end{equation}}
\newcommand{\bee}{\begin{equation*}}
\newcommand{\eee}{\end{equation*}}

\def\p{\partial}

\def\lf{\left}
\def\ri{\right}
\def\Pi{\displaystyle{\mathbb{II}}}

\def\fz{\mathfrak{z}}

\def\H{\mathbb{H}}

\def\a{\alpha}

\def\ii{\sqrt{-1}}

\def\wt{\widetilde }

\def\ol{\overline}

\title[Equation in Relativistic Teichm\"{u}ller Theory]
{On a fully nonlinear equation in Relativistic Teichm\"{u}ller Theory }

\author{Luen-Fai Tam}
\address{Institute of Mathematical Sciences and Department of Mathematics, The Chinese University of Hong Kong}
\email{lftam@math.cuhk.edu.hk}
\thanks{The first author is  partially supported by the Hong Kong RGC General Research Fund \#CUHK 14301517}

\author[Tom Wan]{Tom Yau-heng Wan}
\address{Department of Mathematics, The Chinese University of Hong Kong}
\email{tomwan@cuhk.edu.hk}
\date{\today}


\begin{document}

\begin{abstract}
We  obtain basic estimates for a Monge-Amp\`{e}re equation introduced by Moncrief in the study of the Relativistic Teichm\"{u}ller Theory.  We then give another proof of the parametrization of the Teichm\"uller space obtained by Moncrief. Our approach provides yet another proof of the classical Teichm\"{u}ller theorem that the Teichm\"uller space of a compact oriented surface of genus $g(\Sigma)>1$ is  diffeomorphic  to the disk of dimension $6g(\Sigma)-6$. We also give another proof of properness of a certain energy function on the Teichm\"uller space.

\end{abstract}

\maketitle
\section{Introduction}\label{sec-intro}

In \cite{Moncrief}, Moncrief studied  solutions of the vacuum Einstein equation  on $\Sigma\times \R$ with CMCSH (constant-mean-curvature-spatially-harmonic) gauge, where $\Sigma$ is a compact Riemann surface of genus $g(\Sigma)>1$ with a fixed metric $\rho$ of constant $-1$ scalar curvature. In CMCSH  gauge, each time slice has constant mean curvature so that with the induced metric $g$ the identity map $\mathbf{Id}: (\Sigma,g)\to (\Sigma,\rho)$ is harmonic. In \cite{Moncrief}, it was shown that such a solution to   the vacuum Einstein equation is globally determined by the solution of the following Monge-Amp\`{e}re equation:
\be\label{e-Lambda-intro}
\Delta_\rho u-u+\lf(1+2\lf|\xi\ri|^2\ri)^\frac12=0
\ee
where
\be\label{e-zeta-1-intro}
2\xi_{ab}=2 \mathfrak{z}_{ab}-\lf(2 u_{;ab}-\rho_{ab}\Delta_\rho u\ri).
\ee
Here $\fz$ is a symmetric traceless and divergence free (0,2) tensor (referred as a TT-tensor) on $(\Sigma,\rho)$, and $u_{;ab}$ is the Hessian of $u$ with respect to $\rho$.   We will refer this equation as the {\it Moncrief's equation} in this work.

 It was shown in \cite{Moncrief} that for $\tau<0$, one can find a unique solution $u(\tau)$ corresponding to $\tau \fz$. Using the solutions, we obtain a family of metrics $g(\tau)$. Solving for the lapse $N$ and shift $X$ by \cite{AnderssonMoncrief}, then
$$
ds^2=-N^2dt^2+g_{ab}(dx^a+X^adt)(dx^b+X^bdt)
$$
with $t=-1/\tau$ is a solution to the vacuum Einstein equation on $\Sigma\times(0,\infty)$.
Since harmonicity is preserved under conformal changes of the domain metrics,  Moncrief \cite{Moncrief} shows that one can parametrize the Teichm\"uller space $\mathcal{T}(\Sigma)$ of $\Sigma$ by the space of TT-tensors on $(\Sigma,\rho)$. In \cite{Moncrief}, this was proved by using the Hamilton-Jacobi theory on the cotangent bundle of $\mathcal{T}(\Sigma)$ which is the natural reduced phase space of the vacuum Einstein equations in the form of Hamiltonian dynamical systems, see \cite{Moncrief,Witten}.

In this article, we will obtain various estimates for Moncrief's equation. In particular, we provide   another proof of the above parametrization. In fact, in the original proof of existence of solutions to the equation, it was mentioned that: ``Ideally, this result should follow from estimates derived directly from the Monge-Amp\`{e}re equation," see \cite[p.238]{Moncrief}. More precisely, we first develop estimates to the solutions of the Moncrief's equation and use them to construct a homeomorphism  $\Psi$ from the space of TT-tensors $\mathcal{K}(\rho)$ on $(\Sigma,\rho)$ to $\mathcal{T}(\Sigma)$.  Unlike the Hamilton-Jacobi approach in \cite{Moncrief}, our approach does not assume the Teichm\"{u}ller theorem that $\mathcal{T}(\Sigma)$ is a disk of dimension $6g(\Sigma)-6$, where $g(\Sigma)>1$ is the genus of $\Sigma$. And hence, it provides yet another proof of the Teichm\"{u}ller theorem.

As mentioned in \cite{Moncrief}, this parametrization is complimentary to the parametrization given by M. Wolf in \cite{Wolf}. Unlike Moncrief's construction, Wolf fixed the conformal structure, equivalently the metric $\rho$, on the domain instead of the target when requiring the identity map being harmonic. Then Wolf showed for each holomorphic quadratic differential, equivalently TT-tensor, on $\Sigma$ with respect to the fixed conformal structure on domain surface corresponds uniquely to a conformal structure of $\Sigma$ as target surface via a canonical metric with constant negative curvature. Furthermore, Wolf was able to show that his parametrization gives a compactification of $\mathcal{T}(\Sigma)$ which is equivalent to the Thurston's compactification. The corresponding compactification problem still remains open for Moncrief's construction. The authors intend to study the limiting behaviors of solutions to the Moncrief's equation in the future in order to have a better understanding of the above problem.

The paper is organized as follows: in section \ref{sec-Basic}, we give some basic properties of the Moncrief's equation; in sections \ref{sec-zero-order}, \ref{sec-seond-order}, \ref{sec-higher-order} we derive estimates in various order; in section \ref{sec-Teichmuller} we apply the estimates to obtain the require parametrization and discuss the properness of an energy function on the Teichm\"uller space.

{\it Acknowledgement}: The first author would like to thank Shing-Tung Yau for his suggestion which is the basis for this work and Zhuobin Liang for many useful discussions. The second author would like to thank Richard Schoen, Robert Kusner, and Sumio Yamada for helpful discussions.

\section{Basic    properties of the Moncrief's equation}\label{sec-Basic}
Let
 $(\Sigma,\rho)$ be an oriented compact surface with constant scalar curvature $-1$ metric $\rho$, and let $\mathcal{K}(\rho)$ be the space of all TT-tensors (trace free, divergence free) on $(\Sigma,\rho)$. It is well-known that TT-tensors are in one to one correspondence with holomorphic quadratic differentials on $(\Sigma,\rho)$.

 {\it In the following we always use $\rho$ to raise and lower indices unless specified otherwise.}

 For any $\mathfrak{z}\in \mathcal{K}(\rho)$, we have $\mathfrak{z}_a^b=\rho^{bc}\mathfrak{z}_{ac}$ and $\mathfrak{z}^{ab}=\rho^{ac}\mathfrak{z}^b_c$. Then the (pointwise) norm  of $\mathfrak{z}$ with respect to $\rho$ is given by
\be
|\mathfrak{z}|^2=\rho^{ac}\rho^{bd}\mathfrak{z}_{ab}\mathfrak{z}_{cd}.
\ee

We consider the following fully nonlinear equation obtained by  Moncrief \cite{Moncrief}:
\be\label{e-Lambda-1}
\Delta_\rho u-u+\lf(1+2\lf|\xi\ri|^2\ri)^\frac12=0
\ee
where
\be\label{e-zeta-1}
2\xi_{ab}=2 \mathfrak{z}_{ab}-\lf(2 u_{;ab}-\rho_{ab}\Delta_\rho u\ri).
\ee
Here $u_{;ab}$ is the Hessian of $u$ respect to $\rho$ and the norm of $\xi$ is with respect to $\rho$.
\begin{remark} Note that $\xi$ given by \eqref{e-zeta-1} is symmetric and traceless with respect to $\rho$, namely
$$
\xi_{ab}=\xi_{ba}, \mbox{ } \xi^{ab}=\xi^{ba}, \mbox{ and } \xi^a_a=0.
$$
\end{remark}

We will denote the non-linear term by
\be\label{e-B-def}
B=\lf(1+2\lf|\xi\ri|^2\ri)^\frac12
\ee
and define a metric $g_{ab}$ so that its inverse $g^{ab}$ is given by
\be\label{e-g-1}
 (1+B)    g^{al}=   -2 \xi^{al}+  B \rho^{al}.
   \ee
The fact that $g$ is a metric follows from the Lemma \ref{l-metric} below. It is easy to see that if the solution $u$ of \eqref{e-Lambda-1} is unique for a given $\mathfrak{z}\in\mathcal{K}(\rho)$ (to be proved in Corollary \ref{c-unique}) then all $u, g, B$ depend only on $\fz$. So if needed to be explicit, we will write $u=u(\fz)$, $g=g(\fz)$ etc. Note that $g_{ab}$ and $g^{ab}$ are not related by raising and lowering the indices using metric $\rho$.

\begin{lma}\label{l-metric} Let $u$ be a smooth solution to \eqref{e-Lambda-1}. The tensor $g_{ab}$ defined by equation \eqref{e-g-1}  is a metric with $\mu_g=\sqrt{\det(g)}=\mu_\rho(1+B)$, where $\mu_\rho=\sqrt{\det(\rho)}$. Moreover,
\bee
\rho_{ab}=\frac{B}{1+B}g_{ab}-\frac{2}{1+B}\xi_a^cg_{cb}.
\eee
\end{lma}
\begin{proof} Since $\xi_{ab}$ is symmetric, we may diagonalize $\xi$ with respect to $\rho$ at a point so that $\rho_{ab}=\delta_{ab}$ and $\xi_{ab}=\a_a\delta_{ab}$, then
$$
(1+B)g^{ab}=(-2\a_a+B)\delta^{ab}.
$$
Note that
$$
B^2-4(\a_1^2+\a_2^2)=1+2|\xi|^2-2|\xi|^2=1,
$$
because $\xi$ is trace free. Therefore $g^{al}$ is positive definite and $g_{ab}$ is a metric.

On the other hand, we also have
$$
(1+B)^2\mu_g^{-2}=B^2-2|\xi|^2=1.
$$
Hence $\mu_g=(1+B) $. Since $\mu_\rho=1$ in this coordinates system,  we have $\mu_g=(1+B)\mu_\rho$.

Finally, to express $\rho$ in terms of $g$, we multiply \eqref{e-g-1} by $\rho_{pa}g_{lq}$ and get
\bee
(1+B)\rho_{pq}=-2\xi_p^ag_{aq}+Bg_{pq}.
\eee
The last result follows.
\end{proof}

\begin{lma}\label{l-harmonic} The identity map $\mathbf{Id}:(\Sigma,g)\to (\Sigma,\rho)$ is harmonic, where $g$ is the metric given in the Lemma \ref{l-metric}.
\end{lma}
\begin{proof} Let $\Gamma$ and $\wt\Gamma$ be the connections of $\rho$ and $g$ respectively. Then $\mathbf{Id}$ is harmonic if and only if
$$
V^c=g^{ab}\lf(\Gamma_{ab}^c-\wt\Gamma_{ab}^c\ri)=0.
$$
Taking covariant derivative with respect to $\rho$ in a normal coordinates of $\rho$,
\bee
\begin{split}
{g^{ab}}_{;b}=&{g^{ab}}_{,b}\\
=&-g^{bc}\wt\Gamma_{bc}^a-\frac12g^{as}g^{bc}g_{bc,s}\\
=&V^a-\frac12g^{as}g^{bc}g_{bc,s},
\end{split}
\eee
where ``$;$" denotes covariant derivative with respect to $\rho$ and ``$,$" denotes partial derivative.
On the other hand, by the definition of $g$,
\bee
\begin{split}
\lf[(1+B)g^{ab}\ri]_{;b}=&(-2\rho^{ac} \xi_c^b +B\rho^{ab})_{;b}\\
=&-2\rho^{ac}\xi_{c;b}^b+B_b\rho^{ab}\\
=& \rho^{ac}\rho^{bd}\lf[ 2u_{;cdb}-\rho_{cd}(\Delta_{\rho} u)_b\ri]+B_b\rho^{ab}\\
=&2\rho^{ac}\rho^{bd}u_{;cdb}+\rho^{ab}(B-\Delta_{\rho} u)_b\\
=&2\rho^{ab}(\Delta_{\rho} u)_b-\rho^{ab}u_b+\rho^{ab}(B-\Delta_{\rho} u)_b\\
=&0
\end{split}
\eee
where $B_b$ denotes partial derivative of the function $B$ etc., and we have used that $\mathfrak{z}_{ab}$ is divergence free, $\rho$ has constant Gauss curvature $-\frac12$, and $u$ satisfies \eqref{e-Lambda-1}.
Hence
\bee
\begin{split}
0=&(1+B)^{-1}\lf[(1+B)g^{ab}\ri]_{;b}\\
=& {g^{ab}}_{;b}+g^{ab}[\log(1+B)]_b\\
=&{g^{ab}}_{;b}+g^{ab}[\log\mu_g]_b\\
=&V^a+\frac12g^{as}g^{bc}g_{bc,s}-\frac12g^{as}g^{bc}g_{bc,s}\\
=&V^a
\end{split}
\eee
and $\mathbf{Id}$ is harmonic.
\end{proof}

\begin{lma}\label{l-properties-1} Let $\fz\in \mathcal{K}(\rho)$ and let $g$ be the metric defined in \eqref{e-g-1} via a solution of \eqref{e-Lambda-1}. Let $\lambda$ be a function such that $e^{2\lambda}g$ has constant scalar curvature $-1$. Let  $R(g)$ be scalar curvature of the metric $g$. Denote the harmonic map $\mathbf{Id}:(\Sigma,g)\to (\Sigma,\rho)$ by $w$ and let $\p w$ and $\bar \p w$ be the $\p$ and $\bar \p$-energy densities of $w$ respectively. Then
\bee
\left\{
  \begin{array}{ll}
    |\p w|^2_g=  &\frac12 , \\
    |\bar \p w|_g^2=  &\frac12   -\frac{\mu_\rho}{\mu_g}=\frac12 \frac{B-1}{B+1}, \\
    R(g)=  & -\frac{\mu_\rho}{\mu_g}=-\frac{1}{1+B} < 0, \\
\Delta_g\lambda=&-\frac12 \frac{\mu_\rho}{\mu_g}+\frac12 e^{2\lambda}= \frac12\lf(-\frac1{B+1}+e^{2\lambda}\ri).
  \end{array}
\right.
\eee
\end{lma}
\begin{proof} In a holomorphic coordinates $z=x+iy$ of the conformal class of $[g]$, $g=e^{2\beta}|dz|^2$. Then
\bee
\begin{split}
\rho=&\rho_{11}dx^2+2\rho_{12}dxdy+\rho_{22}dy^2\\
=&\frac14\rho_{11}(dz+  d\bar z)^2-\frac12\ii\rho_{12}(dz+ d\bar z)(dz- d\bar z)-\frac14\rho_{22}(dz-  d\bar z)^2\\
=&\lf(\frac14\rho_{11}-\frac14\rho_{22}-\frac12\ii\rho_{12}\ri)dz^2+\lf(\frac14\rho_{11}-\frac14\rho_{22}+\frac12\ii\rho_{12}\ri)d\bar z^2\\
&+\frac12(\rho_{11}+\rho_{22})dzd\bar z\\
=&\phi dz^2+\ol\phi d\bar z^2+e(\gamma,\rho)\gamma ,
\end{split}
\eee
where
$$
\phi=\frac14\rho_{11}-\frac14\rho_{22}-\frac12\ii\rho_{12}
$$
and $\phi dz^2$ is a holomorphic quadratic differential.

By Lemma \ref{l-metric}, we have
\bee
\rho_{11}-\rho_{22}= -4e^{2\beta}\frac1{1+B}\xi_1^1
\eee
because $\xi_1^1+\xi_2^2=0$. We also have
\bee
\begin{split}
\rho_{12} =&-e^{2\beta}\frac2{1+B}  \xi_1^2.
\end{split}
\eee
Hence
\be
\phi =-e^{2\beta}\frac1{1+B}(\xi_1^1-\ii \xi_1^2)
\ee
and
\be
|\phi|_g^2=\frac1{(1+B)^2}\lf[(\xi_1^1)^2+(\xi_1^2)^2\ri]=\frac14 \frac{B-1}{ B+1} .
\ee
So
\be
|\p w|^2_g|\bar\p w|^2_g=|\phi|_g^2=\frac14 \frac{B-1}{ B+1}.
\ee
On the other hand, by Lemma \ref{l-metric} again,
\be
 |\p w|_g^2+|\bar\p w|_g^2=e^{-2\beta}\frac12(\rho_{11}+\rho_{22})=\frac{B}{B+1}.
\ee
That is,
\bee
\lf(|\p w|_g^2\ri)^2-\frac{B}{B+1}|\p w|_g^2+\frac14 \frac{B-1}{ B+1}=0.
\eee
Hence we have either $|\p w|_g^2=\frac12 $ or $\frac12 \frac{B- 1}{B+1}$ which in turns implies either ``$|\p w|^2_g=\frac12$ and  $|\bar \p w|_g^2=\frac12\frac{B-1}{B+1}$" or ``$ | \p w|_g^2=\frac12\frac{B-1}{B+1}$ and $|\bar\p w|^2_g=\frac12$". Since $|\p w|^2_g-|\bar\p w|_g^2>0$, we conclude that $|\p w|^2_g=\frac12$ and $|\bar \p w|^2=\frac12\frac{B-1}{B+1}$. These are the first two equations.

By Bochner formula
$$
\Delta_g\log |\p w|_g^2=-2K(\rho)(|\p w|_g^2-|\bar\p w|_g^2)+2K(g),
$$
where $K(g)$ is the Gauss curvature of $g$ and $K(\rho)=-\frac12$, the first two results imply
\bee
K(g)=-\frac12\lf(\frac12 -\frac12\frac{B-1}{B+1}\ri)= -\frac12\frac1{B+1}
\eee
and hence we have the third equation
\bee
R(g)=- \frac1{B+1}=-\frac{\mu_\rho}{\mu_g} <0.
\eee
Finally, the last relation follows from the third equation and the fact that $e^{2\lambda}g$ has constant scalar curvature $-1$.

\end{proof}

\section{Zeroth order estimates and uniqueness}\label{sec-zero-order}

In this section, we want to obtain some zeroth order estimates of solutions to \eqref{e-Lambda-1}. We will denote the supnorm of $\fz$ with respect to $\rho$ by $\|\fz\|_\rho$, namely, $\|\fz\|_\rho=\sup_{\Sigma}|\fz|$.

\begin{prop}\label{p-sup-est-1} Let $\mathfrak{z}_1, \mathfrak{z}_2\in  \mathcal{K} (\rho)$, and let $u_1, u_2$ be smooth solutions of \eqref{e-Lambda-1} corresponding to $\fz_1, \fz_2$ respectively. Then
 \begin{enumerate}
   \item [(i)] $$
|u_1-u_2|\le\frac1{\sqrt 2}\|\mathfrak{z}_1-\mathfrak{z}_2\|_\rho.
$$
   \item [(ii)]
\bee
1\le u_1\le 1+\frac1{\sqrt 2}\|\mathfrak{z}_1\|_\rho.
\eee
\item [(iii)] Suppose  $\mathfrak{z}\neq 0$ and $\mathfrak{z}_1=a_1\mathfrak{z}, \mathfrak{z}_2=a_2\mathfrak{z}$, with $a_1, a_2>0$, then
$$
 \lf|\frac{u_1}{a_1}-\frac{u_2}{a_2}\ri|\le \lf|\frac1{a_1}-\frac1{a_2}\ri|.
$$
\item[(iv)] Suppose $\mathfrak{z}_1, \mathfrak{z}_2\neq0$. Let $a_1=\|\mathfrak{z}_1\|_\rho$, $a_2=\|\mathfrak{z}_2\|_\rho$. Assume there is a smooth solution $u_3$ corresponding to $\frac{a_2}{a_1}\fz_1$ (which is a consequence of the existence part of Theorem \ref{t-u}), we have
$$
\lf|\frac{u_1}{a_1}-\frac{u_2}{a_2}\ri|\le\lf|\frac1{a_1}-\frac1{a_2}\ri|+\frac1{\sqrt2} \lf\| \frac{\mathfrak{z}_1}{a_1}-\frac{\mathfrak{z}_2}{a_2}\ri\|_{\rho}.
$$
 \end{enumerate}

\end{prop}
Note: Part (iv) will not be used until the proof of Lemma \ref{l-proper}.
\begin{proof}

\noindent
(i)
Let  $\phi=u_1-u_2$ and define the symmetric tensor $\wt \xi_{ab}$ by
\bee
2\wt \xi_{ab}=2(\mathfrak{z}_1)_{ab} -\lf[2(u_2)_{;ab}-\rho_{ab}\Delta_\rho u_2\ri].
\eee
Then equation \eqref{e-Lambda-1} implies that
\bee
\begin{split}
0=&\Delta_\rho \phi-\phi+\lf(1+2|\xi_1|^2  \ri)^\frac12
-\lf(1+2|\xi_2|^2 \ri)^\frac12\\
=&\Delta_\rho \phi-\phi + \lf[ \lf(1+2|\xi_1|^2  \ri)^\frac12
-\lf(1+2|\wt\xi|^2 \ri)^\frac12 \ri] \\
& \mbox{ } +\lf[ \lf(1+2|\wt\xi|^2  \ri)^\frac12
-\lf(1+2|\xi_2|^2 \ri)^\frac12 \ri]
\end{split}
\eee
where $\xi_1, \xi_2$ are given by \eqref{e-zeta-1} corresponding to $\fz_1, u_1$ and $\fz_2, u_2$ respectively.
Define similarly
$$
2\wt\xi(t)_{ab}= 2(\mathfrak{z}_1)_{ab} -\lf[ 2\lf(tu_2+(1-t)u_1\ri)_{;ab}-\rho_{ab}\Delta_\rho \lf(tu_2+(1-t)u_1\ri) \ri],
$$
then $\wt\xi(0)=\xi_1$ and $\wt\xi(1)=\wt\xi$. Hence
\bee
\begin{split}
\lf(1+2|\xi_1|^2 | \ri)^\frac12-\lf(1+2|\wt\xi|^2 \ri)^\frac12=&
-\int_0^1\frac{\p}{\p t}\lf(1+2|\wt\xi(t)|^2 \ri)^\frac12 dt\\
=&-2\int_0^1\frac{\wt\xi(t)^{ab}\lf(\phi_{;ab}+\frac12\rho_{ab}\Delta_\rho \phi\ri)}{\lf(1+2|\wt\xi(t)|^2 \ri)^\frac12}dt\\
=&-2\lf[\int_0^1\frac{\wt\xi(t)^{ab}  }{\lf(1+2|\wt\xi(t)|^2 \ri)^\frac12}dt\ri]\phi_{;ab}
\end{split}
\eee
because $\wt\xi(t)$ is traceless with respect to $\rho$. Hence, using again the traceless property, we have
$$
0= \lf(\int_0^1\wt{g}^{ab}dt\ri) \phi_{;ab} -\phi +\lf(1+2|\wt\xi|^2  \ri)^\frac12
-\lf(1+2|\xi_2|^2 \ri)^\frac12
$$
where $\wt{g}^{ab}=\rho^{ab}- \frac{2\wt\xi(t)^{ab}  }{\lf(1+2|\wt\xi(t)|^2 \ri)^\frac12}$.
The eigenvalues of $\wt{g}^{ab}$ with respect to $\rho^{ab}$ are
$1 \pm\frac{\lambda}{(1+\lambda^2)^\frac12}$ for $\lambda=\sqrt{2}|\wt\xi(t)|$. Therefore $\wt{g}^{ab}$, and hence $\int_{0}^{1}\wt{g}^{ab}dt$ is positive definite. By maximum principle,
\bee
\begin{split}
\sup_\Sigma\phi\le &\sup_\Sigma\lf[\lf(1+2|\wt\xi|^2  \ri)^\frac12
-\lf(1+2|\xi_2|^2 \ri)^\frac12\ri]\\
=& 2\sup_\Sigma \displaystyle{\frac{|\wt\xi|^2-|\xi_2|^2}{\lf(1+2|\wt\xi|^2  \ri)^\frac12
+\lf(1+2|\xi_2|^2 \ri)^\frac12}} .
 \end{split}
 \eee
 Now
 \bee
 \begin{split}
 |\wt \xi |^2 -| \xi_2|^2
 =&  \wt \xi _a^b  \wt\xi_b^a -(\xi_2)_a^b (\xi_2)_b^a\\
 =&   \lf(\wt\xi_a^b - (\xi_2)_a^b \ri)\, \wt\xi_b^a +
  \lf(\wt\xi_b^a - (\xi_2)_b^a \ri)\, (\xi_2)_a^b \\
 \le &|\mathfrak{z}_1-\mathfrak{z}_2| \lf(|\wt\xi  |+| \xi_2|\ri)
 \end{split}
 \eee
 since $\wt{\xi}-\xi_2=\mathfrak{z}_1-\mathfrak{z}_2$.
 Hence
 \bee
 u_1-u_2=\phi\le \frac{1}{\sqrt 2}\|\mathfrak{z}_1-\mathfrak{z}_2\|_\rho.
 \eee
Similarly
 \bee
 u_2-u_1\le \frac{1}{\sqrt 2}\|\mathfrak{z}_1-\mathfrak{z}_2\|_\rho.
 \eee
From this the first result follows.

\noindent
(ii)
Let $\mathfrak{z}_2=0$, then $u_2=1$ is a solution, in fact unique by Theorem \ref{t-u}. Hence (i) implies
$$
u_1\le 1+ \frac{1}{\sqrt 2}\|\mathfrak{z}_1\|_\rho.
$$
Then applying maximum principle to \eqref{e-Lambda-1}, one concludes also that
$$
u_1\ge 1.
$$
This completes the proof of (ii).

\noindent
(iii) Let $v_1=u_1/a_1,  v_2=u_2/a_2$, and set $\psi=v_1-v_2$. Then
\bee
\begin{split}
0=&\Delta_\rho \psi-\psi + \lf(\frac1{a_1^2} +2\frac{|\xi_1|^2}{a_1^2}  \ri)^\frac12
 - \lf(\frac1{a_2^2}+2\frac{|\xi_2|^2}{a_2^2} \ri)^\frac12\\
=&\Delta_\rho \psi-\psi+\lf(\frac1{a_2^2}+2\frac{|\xi_1|^2}{a_1^2} \ri)^\frac12 -\lf(\frac1{a_2^2}+2\frac{|\xi_2|^2}{a_2^2} \ri)^\frac12\\
& -\lf[\lf(\frac1{a_2^2}+2\frac{|\xi_1|^2}{a_1^2} \ri)^\frac12-\lf(\frac1{a_1^2} +2\frac{|\xi_1|^2}{a_1^2}  \ri)^\frac12\ri].
\end{split}
\eee
As in the proof of (i), interpolate between $\frac{\xi_1}{a_1}$ and $\frac{\xi_2}{a_2}$, we have
$$
\psi\le \sup_\Sigma\lf[ \lf(\frac1{a_2^2}+2\frac{|\xi_1|^2}{a_1^2}  \ri)^\frac12
 -\lf(\frac1{a_1^2}+2\frac{|\xi_1|^2}{a_1^2} \ri)^\frac12\ri].
$$
Similar argument as in (i) again, we can conclude that (iii) is true.

\noindent
(iv) Suppose $\mathfrak{z}_1, \mathfrak{z}_2$ are nonzeros and $\|\mathfrak{z}_1\|_\rho=a_1, \|\mathfrak{z}_2\|_\rho=a_2$. Let $\mathfrak{z}_3=\frac{a_2}{a_1}\mathfrak{z}_1$. By assumption, there is a corresponding solution $u_3$ of \eqref{e-Lambda-1}. Then
\bee
\begin{split}
\sup_\Sigma\lf|\frac{u_1}{a_1}-\frac{u_2}{a_2}\ri|\le &\sup_\Sigma\lf|\frac{u_1}{a_1}-\frac{u_3}{a_2}\ri|
+\frac1{a_2}\sup_\Sigma\lf| u_3  - u_2 \ri|\\
\le& \lf|\frac1{a_1}-\frac1{a_2}\ri|+\frac1{\sqrt2}\frac1{a_2}\lf\| \frac{a_2}{a_1}\mathfrak{z}_1-\mathfrak{z}_2\ri\|_\rho\\
=&\lf|\frac1{a_1}-\frac1{a_2}\ri|+\frac1{\sqrt2}\lf\|  \frac{\mathfrak{z}_1}{a_1}-\frac{\mathfrak{z}_2}{a_2}\ri\|_\rho
\end{split}
\eee
This completes the proof of (iv).
\end{proof}
By (i) in the Proposition \ref{p-sup-est-1}, we have:

\begin{cor}\label{c-unique}
The solution of \eqref{e-Lambda-1} is unique.
\end{cor}
Using Proposition \ref{p-sup-est-1}, we can estimate the area $A(g)$ of $\Sigma$ with respect to the metric $g=g(\fz)$ and the total energy $E(g)$ of the harmonic identity map from $(\Sigma, g)$ to $(\Sigma, \rho)$.
\begin{cor}\label{c-area} Let $\mathfrak{z}\in \mathcal{K}(\rho)$ and let $u=u(\fz)$ and $g=g(\fz)$ be the corresponding solution of \eqref{e-Lambda-1} and metric given by  \eqref{e-g-1} respectively.
\begin{enumerate}
\item [(i)] Let $A(g)$ and $A(\rho)$ be the areas of $\Sigma$ with respect to $g$ and $\rho$ respectively, then
$$
A(g)=A(\rho)+\int_\Sigma u d\mu_\rho,
$$
and
$$
2A(\rho)\le A(g)\le (2+\|\mathfrak{z}\|_\rho)A(\rho).
$$

\item[(ii)] Let $E(g)$ be the total energy of the $\mathbf{Id}: (\Sigma,g)\to (\Sigma,\rho)$. Then
$$
E(g)=A(g)-A(\rho)\le (1+\|\mathfrak{z}\|_\rho)A(\rho).
$$
\end{enumerate}

 \end{cor}
\begin{proof} (i) By Lemma \ref{l-metric} and \eqref{e-Lambda-1},
\bee
\begin{split}
A(g)=&\int_\Sigma (1+B)d\mu_\rho\\
=&A(\rho)+\int_\Sigma u d\mu_\rho .
\end{split}
\eee
The estimates of $A(g)$ then follows from (ii) of Proposition \ref{p-sup-est-1}.

(ii) By Lemma \ref{l-properties-1},
\bee
\begin{split}
E(g)=&\int_\Sigma\lf(1-\frac{\mu_\rho}{\mu_g}\ri)d\mu_g\\
=&A(g)-A(\rho).
\end{split}
\eee
Using (i), we immediately obtain (ii).
\end{proof}
Now we are ready to estimate the injectivity radius and diameter of the metric $g$. We have
\begin{lma}\label{l-geom} Let $\mathfrak{z}\in \mathcal{K}(\rho)$ and $g=g(\fz)$ be the corresponding metric given by \eqref{e-g-1}. Then
\begin{itemize}
\item[(i)] $\rho\le 2g$,
\item[(ii)] $ \mathbf{inj}(g)\ge \frac{1}{\sqrt{2}}\mathbf{inj}(\rho)$, and
\item[(iii)] $\mathbf{diam}(g)\le  \frac{4\sqrt{2}}{\pi \mathbf{inj}(\rho)}\lf( 2+\|\mathfrak{z}\|_\rho \ri) A(\rho)$.
\end{itemize}
\end{lma}

  \begin{proof} (i) In a coordinates chart so that $\rho_{ab}=\delta_{ab}$, $\xi_1^1=\a$, $\xi_2^2=-\a$, and $\xi_1^2=\xi_2^1=0$. Then $|\xi|^2=2\a^2$ and  $g^{12}=g^{21}=0$. Hence
$B=(1+4\a^2)^\frac12$ and
\bee
 g^{11}=\frac{-2\a+(1+4\a^2)^\frac12}{1+(1+4 \a^2)^\frac12}\lf(-2\a+(1+4\a^2)^\frac12\ri)
\le 2.
\eee

Similarly, $ g^{22}\le 2$. Hence $ g^{al}\le 2\rho^{2l}$ and $  g_{al}\ge \frac12\rho_{al}.$

\noindent
{(ii)} Let $x\in\Sigma$ be a point such that $\mathbf{inj}_g(x)=\mathbf{inj}(g)$. Since $R(g)< 0$ (see Lemma \ref{l-properties-1}), there exists a closed geodesic $C$ passing through $x$ with $g$-length $L_g(C)=2\mathbf{inj}(g)$. By $R(g)< 0$ again, $C$ is in fact homotopically nontrivial as a closed curve based at $x$. Otherwise $C$ will be lifted to a closed geodesic in the universal cover which is impossible by the negativity of the curvature.

 Now consider the geodesic ball $B_\rho(x,r)$ centered at $x$ with radius $r=\mathbf{inj}(\rho)$ in the $\rho$-metric. Then $C$ must intersect $\p B_\rho(x,r)$. Otherwise, $C$ is contained in $B_\rho(x,r)$ which is diffeomorphic to a disk in $\R^2$ by the definition of $r=\mathbf{inj}(\rho)$, and hence $C$ is homotopic trivial. This contradicts the construction of $C$. Therefore, the $\rho$-length of $C$ satisfies $L_{\rho}(C) \ge 2r$. Then by part (i), we have
 $$
 2\mathbf{inj}(g)=L_g(C)\ge \frac{1}{\sqrt{2}}L_{\rho}(C)\ge \sqrt{2}r
 $$
 which gives the required estimate.

\noindent
(iii) Let $D=\mathbf{diam}(g)$ and let $r_0=\mathbf{inj}(g)$. Let $m\ge 1$ be the largest integer so that $2m r_0\le D$. Then we can find at least $m$ disjoint geodesic disks of radius $r_0$ in $(\Sigma, g)$. Since $g$ has nonpositive curvature, we have
\bee
A(g)\ge \pi m r_0^2\ge \frac{\pi r}{\sqrt{2}} \cdot\frac{m}{2(m+1)}\cdot  2(m+1)r_0 \ge \frac{\pi r}{4\sqrt{2}}D
\eee
where $r=\mathbf{inj}(\rho)$. By Corollary \ref{c-area}, we have
$$
(2+\|\mathfrak{z}\|_\rho)A(\rho)\ge \frac{\pi r}{4\sqrt{2}}D
$$
which gives the required inequality.
\end{proof}

\section{Second order estimates}\label{sec-seond-order}

Next we want to estimate the Hessian of the solution of \eqref{e-Lambda-1}. The first order estimates follows immediately from the estimate of the Hessian. We always assume that the solution is smooth.

\begin{lma}\label{l-B-equation}
Let $\mathfrak{z}\in \mathcal{K}(\rho)$ be non trivial. Let $g=g(\mathfrak{z})$ and $B=B(\mathfrak{z})$ be the corresponding quantities. Then at the point where $B>1$, we have
$$
\Delta_g B=\frac{2B }{B^2-1} |\nabla_gB|^2-(B-1).
$$
\end{lma}

\begin{proof} Let $\xi=\xi(\mathfrak{z})$. As before, one can check that in a holomorphic coordinates of $g$ so that $g=e^{2\beta}|dz|^2$, the Hopf differential of the harmonic map $\mathbf{Id}: (\Sigma,g)\to (\Sigma,\rho)$ is given by $\phi dz^2$ where
\bee
\phi =-\frac{e^{2\beta}}{1+B}(\xi_1^1-\ii \xi_1^2)
\eee
is holomorphic with
$$
|\phi|^2_g=\frac14\frac{B-1}{B+1}=\frac14(1-\frac2{B+1})=\frac12\frac{|\xi|^2}{(1+B)^2}.
$$
Hence
\bee
\begin{split}
\Delta_g|\phi|_g^2=&-\frac12\Delta_g\lf(\frac1{B+1}\ri)\\
=&\frac12\frac{\Delta_gB}{(B+1)^2}-\frac{|\nabla_gB|^2}{(B+1)^3}.
\end{split}
\eee
On the other hand,
\bee
\Delta_g\log |\phi|_g^2=\Delta_g e^{-4\beta}=2R(g)
\eee
because $\Delta_g\log |\phi|^2=0$ as $\phi$ is holomorphic. Therefore
\bee
\begin{split}
\Delta_g|\phi|_g^2=&\Delta_g(e^{-4\beta}|\phi|^2)\\
=&-4e^{-4\beta}|\phi|^2\Delta\beta+e^{-4\beta}\Delta_g|\phi|^2\\
=& 2|\phi|^2_gR(g)+4|\phi_z|^2\\
=&-\frac12\frac{B-1}{(B+1)^2}+\frac{|\nabla_g |\phi|^2_g|^2}{|\phi|_g^2}\\
=&-\frac12\frac{B-1}{(B+1)^2}+\frac14\cdot\frac{|\nabla_g B|^2}{(B+1)^4}\cdot \frac{4(B+1)}{B-1}\\
=&-\frac12\frac{B-1}{(B+1)^2}+ \frac{|\nabla_g B|^2}{(B+1)^2(B^2-1)},
\end{split}
\eee
where we have used Lemma \ref{l-properties-1}.
So
\bee
\begin{split}
\Delta_g B=& \frac{2|\nabla_gB|^2}{B+1}-(B-1)+\frac{2(B+1)^2|\nabla_gB|^2}{(B+1)^2(B^2-1)}\\
=& \frac{2B }{B^2-1} |\nabla_gB|^2-(B-1).
\end{split}
\eee
\end{proof}
\begin{lma}\label{l-B-grad}
Let $\mathfrak{z}$, $B, g$ as in Lemma \ref{l-B-equation}. Then there exists a constant $C>0$ independent of $\fz$ such that
$$
|\nabla_g\log B|\le C.
$$
\end{lma}
\begin{proof}
Let $h=|\nabla_g\log B|^2$. By Lemma \ref{l-B-equation} at the point where $B>1$,

\be\label{e-B}
\begin{split}
\Delta_g \log B=&\frac{\Delta_g B}{B}-|\nabla_g \log B|^2\\
=&\frac1B\lf(\frac{2B}{B^2-1}|\nabla_g B|^2   -  (B-1)\ri)-h\\
=&\lf(\frac{2B^2}{B^2-1}-1\ri)h-\frac{B-1}{B}\\
=&\frac{B^2+1}{B^2-1}h-\frac{B-1}{B }.
\end{split}
\ee
At a maximum point $p$ of $h$, we may assume $B>1$. Otherwise, $B\equiv 1$ and $\fz$ is trivial. Then in an orthonormal  frame $e_i$ at $p$ with respect to  $g$,
\be
0=h_i=2 (\log B)_k (\log B)_{ki}
\ee
for $i=1, 2$. Here we have denoted the covariant derivatives of a function $f$ with respect to $g$ by $f_i$, $f_{ij}$ and $f_{ijk}$ etc. This convention is just for the proof of this lemma in order to simplify notations.
Then at the point $p$,
\bee
\begin{split}
0\ge &\Delta_g h\\
=&\sum_i h_{ii}\\
=&2 (\log B)_{ki} (\log B)_{ki}+2\lf(\log B\ri)_k (\log B)_{kii}\\
=&2\sum_{ki}(\log B)_{ki}^2+2\lf(\log B\ri)_k (\log B)_{iik}+2 R_{ik}(\log B)_i(\log B)_k\\
=&2\sum_{ki}\lf((\log B)_{ki}\ri)^2+2\lf(\log B\ri)_k \lf(\frac{B^2+1}{B^2-1}h-\frac{B-1}{B }\ri)_{k}-\frac{1}{B+1} h
\end{split}
\eee
because $R(g)=-1/(B+1)$ by Lemma \ref{l-properties-1}. We may assume that $h>0$ at $p$. One may choose orthonormal frame such that $e_1=\nabla_g B/|\nabla_gB|$ and $e_2\perp e_1$. So $B_2=0$. Then
$$
0=h_1=2(\log B)_1 (\log B)_{11} \mbox{ and } 0=h_2=2(\log B)_1(\log B)_{21}.
$$
Hence $(\log B)_{11}=(\log B)_{12}=0$ and so
\bee
\sum_{ki}\lf((\log B)_{ki}\ri)^2=(\Delta_g\log B)^2.
\eee
On the other hand, since $h_k=0$,
\bee
\begin{split}
\lf(\log B\ri)_k \lf(\frac{B^2+1}{B^2-1}h-\frac{B-1}{B }\ri)_{k}=&\lf(\log B\ri)_k\lf(-\frac{4B}{(B^2-1)^2}h-\frac{ 1}{B^2 }\ri)B_k\\
=&h\lf(-\frac{4B^2}{(B^2-1)^2}h-\frac1B \ri)
\end{split}
\eee
Hence we have at $p$
\bee
\begin{split}
0\ge &2\lf(\frac{B^2+1}{B^2-1}h-\frac{B-1}{B }\ri)^2+2h\lf(-\frac{4B^2}{(B^2-1)^2}h
-\frac1B \ri)
-\frac{1}{B+1}h.
\end{split}
\eee
Multiplying both sides by $B^2(B^2-1)^2$, we have
\bee
\begin{split}
0\ge &2\lf[B(B^2+1)h-  (B-1)^2(B+1) \ri]^2+2h\lf[-4B^4h-  B(B^2-1)^2  \ri]\\
&
-  B^2 (B-1)^2(B+1)h\\
\ge &2h^2\lf(B^6-2B^4+B^2\ri)\\
 &-h(B-1)^2(B+1)  \lf[4B(B^2+1)+2B(B+1) +B^2\ri] \\
=&h  B(B-1)^2 (B+1)\lf[2h B(B+1) - \lf( 4 (B^2+1)+2(B+1)+B\ri)\ri]
\end{split}
\eee
  Hence $h\le C_1 $ for some absolute constant $C_1$ independent of $\mathfrak{z}$ because $B\ge 1$.
\end{proof}

\begin{prop}\label{p-B-bound} Let $\fz\in \mathcal{K}(\rho)$, $B$ as in Lemma \ref{l-B-grad}. Then there is a constant $C>0$ independent of $\fz$ such that
$B\le \exp\lf[C(1+\|\fz\|_\rho)\ri]$. Moreover, if $u=u(\fz)$ is the corresponding solution to \eqref{e-Lambda-1}, then $|\nabla_\rho u|\le C$ and $|\nabla_\rho^2 u|\le C$ for some constant $C$ depending only on $\|\fz\|_\rho$ and $\rho$.

\end{prop}
\begin{proof} The first part of the corollary follows from (iii) of Lemma \ref{l-geom}, Lemma \ref{l-B-grad} and the fact that there is always a point on $\Sigma$ such that $\xi=0$. This is because $\xi$ can be regraded as a holomorphic quadratic differential with respect to the conformal structure given by $g$. Hence $B=1$ at that point.

To prove the second statement, by Proposition \ref{p-sup-est-1}, we have
$$
|u|\le 1+ \frac1{\sqrt2}\|\fz\|_\rho.
$$
Hence by \eqref{e-Lambda-1} and the defintition of $B$, we have
$$
|\Delta_\rho u|\le C
$$
for some constant $C $ depending only on an upper bound of $\|\fz\|_\rho$ and $\rho$. By \eqref{e-zeta-1}, it is then easy to see that $|\nabla_\rho^2 u|\le C$ for some constant $C$ depending only on the upper bound of $\|\fz\|_\rho$ and $\rho$. Since $\nabla_\rho u=0$ somewhere, we conclude that  $|\nabla_\rho u|\le C$ for some constant $C$ depending only on the upper bound of $\|\fz\|_\rho$ and $\rho$.

\end{proof}

\section{Higher order estimates and existence}\label{sec-higher-order}

Using the second order estimate, it is rather standard to obtain higher order estimates. Namely we have the following:

\begin{prop}\label{p-u-derivatives} Let $\fz\in \mathcal{K}(\rho)$ and let $u=u(\fz)$ be the solution of \eqref{e-Lambda-1}.
Suppose $||\fz||\le \kappa$. Then for any $k\ge 2$, there is a constant $C$ depending only on $\kappa, k, \rho$  such that $|\nabla_\rho^ku|\le C$.
\end{prop}

Before we prove the proposition, we have the following setup:  Let $\xi=\xi(\fz)$ be the corresponding tensor in \eqref{e-zeta-1}. Then any geodesic disk of radius $r<\mathbf{inj}(\rho)$ is isometric to a geodesic ball of radius $r$ in $\H^2$.  Hence the metric $\rho$ is of the form
$$
\rho=e^{2f}(dx^2+dy^2).
$$
for $x^2+y^2<r^2$. In this geodesic ball, $\rho_{ab}=e^{2f}\delta_{ab}$ implies
\bee
\begin{split}
\Gamma_{ab}^c=&\frac12 \rho^{cd}(\rho_{ad,b}+\rho_{db,a}-\rho_{ab,d})\\
=&\frac12e^{-2f}(\rho_{ac,b}+\rho_{cb,a}-\rho_{ab,c})\\
=& f_b\delta_{ac}+f_a\delta_{bc}-f_c\delta_{ab}
\end{split}
\eee
and hence the Hessian of $u$ is given by
\bee
\begin{split}
u_{;ab}=&u_{,ab}-(f_b\delta_{ac}+f_a\delta_{bc}-f_c\delta_{ab})u_c\\
=&u_{,ab}-(f_bu_a+f_au_b-\delta_{ab}f_cu_c),
\end{split}
\eee
where $u_{;ab}$ is covariant derivative; $u_a$ and $u_{,ab}$ are first and second order partial derivatives etc., and $u_1=u_x, u_2=u_y$ etc.
Therefore
\be\label{e-nabla-u-1}
\left\{
  \begin{array}{ll}
    u_{;11}= &u_{,11}-(f_1u_1-f_2u_2) \\
    u_{;22}= &u_{,22}-(f_2u_2-f_1u_1) \\
    u_{;12}= &u_{,12}-(f_1u_2+f_2u_1).
  \end{array}
\right.
\ee
Note that in this notation,
\bee
\begin{split}
2 |\xi|^2=&2|\fz|^2-4 e^{-2f}  \xi_a^bu_{;ab}+e^{-4f}\lf(u_{;11}^2+u_{;22}^2+4u_{;12}^2-2u_{;11}u_{;22}\ri).
\end{split}
\eee
Since the trace of $\xi$ is zero, we further let $\a=\xi_1^1=-\xi_2^2$ and $\beta=\xi_1^2$. Then the equation \eqref{e-Lambda-1} can be written as
\be\label{e-Lambda-local-1}
F:=(u_{;11}+u_{;22})-e^{2f}u +\bigg[e^{4f}+X^2+Y^2\bigg]^\frac12 = 0
\ee
where
$$
X=-2e^{2f}|\fz|\a+ (u_{;11}-u_{;22}) \mbox{ and }  Y=-2(e^{2f} \beta+ u_{;12}).
$$

Following standard notations in fully nonlinear PDE theory, we let $p=u_x, q=u_y, r=u_{xx}, s=u_{xy}, t=u_{yy}$. Then the equation is of the form $$
F(x, y, u, p,q,r,s,t)=0
$$
with
\bee
\left\{
  \begin{array}{ll}
   F_r=&1+X\lf[e^{4f}+X^2+Y^2\ri]^{-\frac12}; \\
     F_t=&1-X\lf[e^{4f}+X^2+Y^2\ri]^{-\frac12}; \\
    F_s=&Y\lf[e^{4f}+X^2+Y^2\ri]^{-\frac12}
  \end{array}
\right.
\eee
The following lemma show that it is elliptic.
\begin{lma}\label{l-estimates-1} Let $u, \fz$ as in Proposition \ref{p-u-derivatives}.  Then there are positive constants $C_i$, $i=1,2$, depending only on $\kappa, \rho$ and $\|\fz\|_\rho$ such that  at $u$,
$$
0<C_1\le F_r, F_t\le 2 \mbox{ and } |F_s|\le 1,
$$
and
$$
0<C_2\lf[(v^1)^2+(v^2)^2\ri]\le F_r(v^1)^2+ F_s v^1 v^2+F_t(v^2)^2\le 3\lf[(v^1)^2+(v^2)^2\ri]
$$
for any $v^1,v^2\in \R$.
\end{lma}
\begin{proof} In the following $C_i$ will denote a positive constant depending only on $\kappa$ and $\rho$.

It is easy to see that $F_r, F_t\le 2$, and $|F_s|\le 1$ which implies
$$
F_r(v^1)^2+ F_sv^1v^2+F_t(v^2)^2\le 3\lf[(v^1)^2+(v^2)^2\ri].
$$
Since $|\nabla_\rho^2u|\le C$ for some constant $C$ depending only on $\kappa, \rho$ by Proposition \ref{p-B-bound}, and there is a point where $\nabla_\rho u=0$,  we have $|\nabla_\rho u|\le C_3$. Hence $|X|\le C_4$ and so $F_r, F_s\ge C_1>0$.
On the other hand, at $u$,
$$
F_rF_t-\frac14F_s^2\ge 1-(X^2+Y^2)\lf[e^{4f}+X^2+Y^2\ri]^{-1}\ge C_4>0.
$$
Hence
\bee
\begin{split}
F_r(v^1)^2+ F_sv^1v^2&+F_t(v^2)^2 \\
=&F_r \lf((v^1)^2+ \frac{F_s}{F_r}v^1v^2+\frac{F_s^2}{4F_r^2}(v^2)^2\ri)+\frac{F_rF_t-\frac14F_s^2}{F_r}(v^2)^2\\
\ge&\frac{C_4}2(v^2)^2.
\end{split}
\eee
Similarly, one can prove that
$$
F_r(v^1)^2+2F_sv^1v^2 +F_t(v^2)^2\ge \frac{C_4}2(v^1)^2.
$$
This completes the proof of the lemma.
\end{proof}

With the ellipticity, we have
 \begin{lma}\label{l-estimates-3}
 There is $\delta>0$ depending on the quantities mentioned in the Lemma \ref{l-estimates-1} such that if $u\in C^3$, then the $C^{2,\delta}$ is bounded by a constant depending on the quantities mentioned in the  Lemma \ref{l-estimates-1}.
 \end{lma}
\begin{proof} Let $v=u_x$.   Differentiating \eqref{e-Lambda-local-1} with respect to $x$, say, and let $v=u_x$, we see that  $v$  satisfies
 $$
 0=F_rv_{xx}+F_s v_{xy}+F_t v_{yy}+F_p u_{xx}+F_qu_{yx}+F_u u_x+F_x.
 $$
 Since $|\nabla_\rho^2 u|$ is uniformly bounded, one can apply \cite[Theorem 12.4]{GilbargTrudinger} to get the result. The estimate for $u_y$ is similar.
\end{proof}

With the $C^{2,\delta}$ bound, we can develop the higher order bounds and the Proposition \ref{p-u-derivatives} follows immediately from the following
 \begin{lma}\label{l-estimates-2}
 Suppose $u\in C^{k,\delta}$ for $k\ge 2$ and $u$ is at least $C^3$. Then $u\in C^{k+1,\delta}$ so that its $C^{k+1,\delta}$ norm is bounded by a constant depending only on $C^{k,\delta}$ norm of $u$, $ \rho$ and the upper bound of $\|\fz\|_\rho$.
 \end{lma}
 \begin{proof} Differentiating \eqref{e-Lambda-local-1} with respect to $x$, say, and let $v=u_x$ as before, we see that  $v$  satisfies
 $$
 0=F_rv_{xx}+F_s v_{xy}+F_t v_{yy}+F_p u_{xx}+F_qu_{yx}+F_u u_x+F_x.
 $$
 Suppose $u\in C^{k,\delta}$, $k\ge 2$, then $F_r, F_s, F_t$ are in $C^{k-2,\delta}$ with norms bounded by constants depending only on the quantities mentioned in the lemma. Similar for
 $$
 F_p u_{xx}+F_qu_{yx}+F_u u_x+F_x.
 $$
 By Lemma \ref{l-estimates-1} and \cite[Cor 6.3,Th. 6.19, ex. 6.1]{GilbargTrudinger}, using the fact that $|u|\le C_1$ we conclude that $v_x\in C^{k,\delta}$ with norm bounded by a constant depending only on the quantities mentioned in the lemma. From this the result follows.
\end{proof}

With Proposition \ref{p-u-derivatives}, we can now reprove the following result of \cite{Moncrief} without using Hamilton-Jacobi theory and hence avoid the use of Teichm\"uller theorem.
\begin{thm}\label{t-u}  For any $\mathfrak{z}\in\mathcal{K}(\rho)$ and $\tau\in(-\infty, 0)$, there is a unique smooth solution $u$ of
\bee
\Delta_\rho u-u+\lf(1+2\tau^2\lf|\xi\ri|^2\ri)^\frac12=0.
\eee
where
\bee
2|\tau|\xi_{ab}=2|\tau|\mathfrak{z}_{ab}-\lf(2 u_{;ab}-\rho_{ab}\Delta_\rho u\ri),
\eee
and the norm is with respect to $\rho$.
\end{thm}
\begin{proof} Uniqueness follows from Corollary \ref{c-unique}. To prove existence, in  section 8 of \cite{Moncrief}, it was first proved, without using the Hamilton-Jacobi theory that there is $\tau_0<0$ such that there is a solution of $\tau \in ( \tau_0, 0)$. In order to extend the solution for all $\tau<0$, it is clear that one can now use the estimates obtained in sections \ref{sec-zero-order}, \ref{sec-seond-order} and Proposition \ref{p-u-derivatives} instead of the Hamilton-Jabobi theory.
\end{proof}

\section{Applications to Teichm\"uller theory}\label{sec-Teichmuller}

Consider the Teichm\"uller space $\mathcal{T}=\mathcal{T}(\Sigma)$    of $\Sigma$ .
Let $\mathcal{M}_{-1}$ be the space of metrics with constant scalar curvature $-1$ on $\Sigma$ and let $s$ be the section
$$
s: \mathcal{T}\to  \mathcal{M}_{-1}
$$
constructed as follows: For any $[h]\in \mathcal{T} $, let the representative $h$ of the class $[h]$ be the unique metric in the class with constant scalar curvature $-1$.  Let $u: (\Sigma,h)\to(\Sigma, \rho)$ be the harmonic diffeomorphism isotopic to the identity and set $\gamma=u_*(h)$. Then $\mathbf{Id}: (\Sigma,\gamma)\to(\Sigma, \rho)$ is harmonic. Note that $\gamma$ depends only on $[h]$. Then the section $s:\mathcal{T}\to \mathcal{M}_{-1}$ is defined by $s([h])=\gamma$. We identify $\mathcal{T}$ with $s(\mathcal{T})$. (See \cite{Tromba} for detail and note that the section $s$ depends on $\rho$ and $[\rho]$ can be considered as a base point in $\mathcal{T}$.)

We are going to define a map
\be
\Psi: \mathcal{K}(\rho)\to \mathcal{T}
\ee
using this identification. The map $\Psi$ is defined as follows:   assigning $\mathfrak{z}\in\mathcal{K}(\rho)$ to the unique metric $\gamma=s([g])$ with constant scalar curvature $-1$ in the conformal class of the metric $g$ given by Lemma \ref{l-metric}. That is $\Psi(\mathfrak{z})=\gamma=e^{2\lambda}g$ where $\lambda$ is a function such that $R(\gamma)=-1$. Note that $\lambda$ can be found by solving uniquely a semi-linear equation on $\Sigma$, see Lemma \ref{l-properties-1}. This gives a well-defined map $\Psi$ from $\mathcal{K}(\rho)\to \mathcal{T}$ and one has
\begin{thm}[Moncrief]\label{t-homeomorphism}
The map $\Psi: \mathcal{K}(\rho) \to \mathcal{T}$ is a  diffeomorphism.
\end{thm}
Note that the proof of bijectivity of $\Psi$ given in \cite{Moncrief} used Hamilton-Jacobi theory which need the Teichm\"uller theorem (see Corollary  \ref{c-Teichmuller} below) that $\mathcal{T}$ is homeomorphic to $\R^{6g(\Sigma)-6}$. Using results in previous sections, we will give another proof of the theorem without using the Teichm\"uller theorem. Hence as a corollary, we have:
\begin{cor}[Teichm\"uller Theorem]\label{c-Teichmuller}
The Teichm\"uller space $\mathcal{T}$ of a compact surface $\Sigma$ with genus $g(\Sigma)>1$ is  diffeomorphic  to $\R^{6g(\Sigma)-6}$.
\end{cor}
\begin{proof}
Note that the space $\mathcal{K}(\rho)$ of TT-tensors with respect to $\rho$ can be identified as the space of holomorphic quadratic differentials on $\Sigma$ with complex structure given by the conformal class of $\rho$. Then Riemann-Rock Theorem and Theorem \ref{t-homeomorphism} give the required result.
\end{proof}

\subsection{$\Psi$ is injective}\label{ss-injective}

We first prove that the map $\Psi$ is one-to-one using standard maximum principle.
\begin{lma}\label{l-oneone} The map
$\Psi$ is injective.
\end{lma}
\begin{proof} Let $\mathfrak{z}_1, \mathfrak{z}_2\in K(\rho)$ and let $u_1$, $u_2$ be the corresponding solutions to \eqref{e-Lambda-1}.  Then let  $\xi_1, \xi_2$  be given by \eqref{e-zeta-1}, and $g_1, g_2$ be given by \eqref{e-g-1}. Let $\gamma_i=\Psi(\mathfrak{z}_i)$ be the metrics with constant scalar curvature $-1$ so that $\gamma_i=e^{2\lambda_i}g_i$. Suppose $\gamma_1=\gamma_2=\gamma$. We want to prove that $\mathfrak{z}_1=\mathfrak{z}_2$. By Lemma \ref{l-properties-1}, for $i=1, 2$,
$$
\Delta_{g_i} \lambda_i=-\frac12\frac{\mu_\rho}{\mu_{g_i}} +\frac12e^{2\lambda_i}.
$$
Dividing the equation by $e^{2\lambda_i}$, we have
$$
\Delta_{\gamma_i}\lambda_i=-\frac12\frac{\mu_\rho}{\mu_{\gamma_i}} +\frac12.
$$
Since $\gamma_1=\gamma_2=\gamma$, we have
$$
\Delta_\gamma\lambda_i=-\frac12\frac{\mu_\rho}{\mu_{\gamma}} +\frac12
$$ and hence
\bee
\begin{split}
0=&\Delta_\gamma(\lambda_1-\lambda_2).
\end{split}
\eee
Then standard maximum principle implies that $\lambda_1=\lambda_2+C$ for some constant $C$ and hence $g_1=\sigma^2g_2$ for some constant $\sigma>0$. By Lemma \ref{l-metric},
$$
\rho_{ab}=\frac{B_i}{1+B_i}(g_i)_{ab}-2\frac1{1+B_i}(\xi_i)_a^c(g_i)_{cb}.
$$
$i=1,2$, where
$$
B_i=\lf(1+2|\xi_i|^2\ri)^\frac12.
$$
Using $g_1=\sigma^2g_2$, we have
\bee
\lf(\frac{B_1}{1+B_1}\sigma^2-\frac{B_2}{1+B_2}\ri)(g_2)_{ab}=\lf(\frac{2}{1+B_1}\sigma^2(\xi_1)_a^c- \frac{2}{1+B_2}(\xi_2)_a^c\ri) (g_2)_{cb}.
\eee
 Taking trace with respect to $g_2$ to both sides implies
 \bee
 2\lf(\frac{B_1}{1+B_1}\sigma^2-\frac{B_2}{1+B_2}\ri)=2\lf(\sigma^{2}\frac1{1+B_1}(\xi_1)_a^a - \frac1{1+B_2}(\xi_2)_a^a \ri)=0.
 \eee
 Hence by Lemma \ref{l-metric} again, we have
 \bee
 \sigma^2B_1\frac{\mu_\rho}{\mu_{g_1}}=B_2\frac{\mu_\rho}{\mu_{g_2}}
 \eee
 and so $B_1=B_2$ because $\sigma^2\mu_{g_1}^{-1}=\mu_{g_2}^{-1}.$ Putting it back to the previous equation, this implies that $\sigma=1$ and hence $g_1=g_2$ and $\xi_1=\xi_2$ by Lemma \ref{l-metric}. Finally, applying maximum principle to \eqref{e-Lambda-1}, we have $u_1=u_2$. Therefore, $\mathfrak{z}_1=\mathfrak{z}_2$ by \eqref{e-zeta-1}. This completes the proof of the injectivity of $\Psi$.
\end{proof}

\subsection{$\Psi$ is surjective}\label{ss-surjective}
Next, we show the surjectivity.
\begin{lma}\label{l-onto}
The map $\Psi$ is surjective.
\end{lma}
\begin{proof} By the identification of $\mathcal{T}$ as $s(\mathcal{T})$, we need to show that for any $\gamma\in \mathcal{M}_{-1}$ so that $\mathbf{Id}: (\Sigma,\gamma)\to (\Sigma,\rho)$ is harmonic, we can find a TT-tensor $\mathfrak{z}\in\mathcal{K}(\rho)$ such that $\Psi(\mathfrak{z})=\gamma$.

As in the proof of Lemma \ref{l-properties-1}, one sees that in a holomorphic coordinates such that  $\gamma_{ab}=e^{2\beta}|dz|^2$,
\bee
\rho=\phi dz^2+\ol\phi d\ol{z}^2 +e(\gamma,\rho)\gamma
\eee
where
$$
\phi=u+\sqrt{-1}v = \frac14\lf( \rho_{11}- \rho_{22}\ri) -\frac12\ii\rho_{12}
$$
is holomorphic, and $e(\gamma,\rho)$ is the energy density with respect to the metrics $\gamma$ and $\rho$. Let
$$
\hat k=-\frac{1}{2} \lf[ \phi dz^2+\ol\phi d\ol{z}^2 \ri].
$$
Then the holomorphicity of $\phi$ implies that $\hat k$  is a $TT$-tensor with respect to $\gamma$. And the metric $\rho$ can be expressed as
\bee
\rho=-2\hat k+ e(\gamma,\rho)\gamma.
\eee
Let $\lambda$ be a function defined by
$$
e^{-2\lambda}=e(\gamma,\rho)+\frac{\mu_\rho}{\mu_\gamma}.
$$
Since
\bee
\begin{split}
e(\gamma,\rho)^2-
2|\hat k|^2_{\gamma}=&e^{-4\beta}\lf[\frac14(\rho_{11}+\rho_{22})^2-2( 2u^2+2v^2)\ri]\\
=& e^{-4\beta}\lf[\frac14\lf(\rho_{11}+\rho_{22}\ri)^2-4\lf(\frac14\rho_{11}-\frac14\rho_{22}\ri)^2-\rho_{12}^2\ri]\\
=&e^{-4\beta}\lf(\rho_{11}\rho_{22}-\rho_{12}^2\ri)\\
=&\frac{\mu_\rho^2}{\mu_\gamma^2}>0,
\end{split}
\eee
$\lambda$ satisfies
\be\label{e-surjective-1}
e^{-4\lambda}-2e^{-2\lambda}e(\gamma,\rho)+2|\hat k|_\gamma^2=0.
\ee

Let $g=e^{-2\lambda}\gamma$, and let $k=\hat k+\frac12 g$. Then
$$
|k|^2_g=|\hat k|^2_g+\frac12 =e^{4\lambda}|\hat k|_\gamma^2+\frac12 =e^{2\lambda}e(\gamma,\rho).
$$
Putting it in the expression of $\rho$ in terms of $\hat k$ and $\gamma$, we have
\be\label{e-rho}
\rho=-2\hat k+ |k|_g^2g.
\ee
Now, one can define $\xi_{a}^b$ by
$$
\hat k_{ab}=\frac{\mu_\rho}{\mu_g}\xi_a^cg_{cb}
$$
or more explicitly,
$$
\xi_a^c=\frac{\mu_g}{\mu_\rho}g^{bc}\hat k_{ab}.
$$
As before, we define $\xi_{ab}=\rho_{ac}\xi_b^c$ using metric $\rho$. Then it is clear from the definition that $\xi$ is trace free with respect to $\rho$ and can be expressed in terms of $\hat k$ and $g$ as follows:
\bee
\begin{split}
\xi_{ab}=&\rho_{ac}\xi_b^c\\
=&\frac{\mu_g}{\mu_\rho}g^{cd}\hat k_{bd}\lf(|k|_g^2g_{ac}-2\hat k_{ac}\ri)\\
=&\frac{\mu_g}{\mu_\rho}\lf(|k|_g^2\hat k_{ab}-2g^{cd}\hat k_{bd}\hat k_{ac}\ri).
\end{split}
\eee
Hence $\xi_{ab}$ is symmetric in $a,b$. Let
$$
B=\lf(1+2|\xi|^2\ri)^\frac12,
$$
where
\bee
|\xi|^2=\xi_a^b\xi_b^a
=\frac{\mu_g^2}{\mu_\rho^2}g^{bc}\hat k_{ca}g^{ad}\hat k_{bd}
=\frac{\mu_g^2}{\mu_\rho^2}|\hat k|_g^2.
\eee
On the other hand, by \eqref{e-rho} we have
$$
\frac{\mu_\rho^2}{\mu_g^2}=\lf(\frac12-|\hat k|_g^2\ri)^2.
$$
So
\bee
\begin{split}
\frac{\mu_\rho}{\mu_g}B=&\lf(\frac{\mu_\rho^2}{\mu_g^2}+2|\hat k|_g^2\ri)^\frac12\\
=&\lf[\lf(\frac12-|\hat k|_g^2\ri)^2+2|\hat k|_g^2\ri]^\frac12\\
=&\frac12+|\hat k|_g^2\\
=&|k|_g^2.
\end{split}
\eee
Hence
\bee
\frac{\mu_\rho^2}{\mu_g^2}+2|\hat k|_g^2  = |k|_g^4.
\eee
That is
\bee
(1-\frac1{B^2})|k|_g^4-2|k|_g^2+1=0.
\eee
Hence
$ |k|_g^2=\frac{B}{B+1}$ or $\frac{B}{B-1}$. Note that $\hat k$ must be zero somewhere, and hence $\xi=0$ and $B=1$ somewhere. Hence we must have $ |k|_g^2=\frac{B}{B+1}$. This implies
 \be\label{e-rho-g}
 \left\{
   \begin{array}{rll}
    (1+B)g^{al}& =  & -2\rho^{ab}\xi_b^l+B\rho^{al}; \\
     \mu_g &=&(1+B)\mu_\rho .
   \end{array}
 \right.
 \ee
Next, let $u$ be the solution of
$$
\Delta_\rho u-u+(1+2|\xi|^2)^\frac12=0,
$$
which existence and uniqueness are ensured by standard theory of elliptic PDE,
and define $\mathfrak{z}_{ab}$ by
\bee
2\xi_{ab}=2\mathfrak{z}_{ab}-(2u_{;ab}-\rho_{ab}\Delta_\rho u).
\eee
We claim that $\mathfrak{z}_{ab}\in \mathcal{K}(\rho)$. Then one can see that $\Psi(\mathfrak{z})=\gamma$ and $\Psi$ is surjective.

To prove the claim, it is easy to see that $\mathfrak{z}_{ab}$ is symmetric. Taking trace with respect to $\rho$ in the defining formula for $\mathfrak{z}_{ab}$, we have
\bee
2\rho^{ab}\mathfrak{z}_{ab}=2\rho^{ab}\xi_{ab}
=2\xi_a^a
=0.
\eee
So $\mathfrak{z}$ is traceless. From the proof of Lemma \ref{l-harmonic}, using the assumption that $\mathbf{Id}:(\Sigma, \gamma) \to (\Sigma, \rho)$ is harmonic, we have
$$
0=-2\rho^{ac}\mathfrak{z}_{c;b}^b
$$
in a normal coordinate neighborhood of $\rho$. This proves the claim and completed the proof of the lemma.
\end{proof}

\subsection{$\Psi$ is continuous}\label{ss-continuous}

 Next we show the continuity of $\Psi$ which is a consequence of the $C^{k,\delta}$ estimates of all $k\ge 0$.
\begin{lma}\label{l-continuous} The map $\Psi$ is continuous.
\end{lma}
\begin{proof} Let $\mathfrak{z}_n\in \mathcal{K}(\rho)$ be a sequence such that $\mathfrak{z}_n\to \mathfrak{z}\in \mathcal{K}(\rho)$. Let $u_n, u, g_n, g, B_n, B$ be the corresponding quantities as in \eqref{e-Lambda-1} and \eqref{e-g-1}, and $\gamma_n=e^{2\lambda_n}g_n$ with $R(\gamma_n)=-1$. In particular, convergence of $\mathfrak{z}_n$ implies that there is a $\kappa>0$ such that $|\mathfrak{z}_n|\le \kappa$ for all $n$. By Proposition \ref{p-B-bound},   $|\nabla_\rho^2 u_n|\le C_1$ for some $C_1>0$ independent of $n$. Hence $u_n$ will subconverge to $u$ in $C^\infty$ norm by the Proposition \ref{p-sup-est-1} and Proposition \ref{p-u-derivatives}.   This implies that $u_n\to u$ in $C^\infty$ norm. Hence $g_n\to g$ in $C^\infty$ norm. Then by Lemma \ref{l-properties-1}, we have $\Psi(\mathfrak{z}_n)=\gamma_n$ converges in $C^\infty$ norm to $\gamma=e^{2\lambda}g$ with $R(\gamma)=-1$. This completes the proof of the lemma.
\end{proof}

\subsection{$\Psi^{-1}$ is continuous}\label{ss-continuous-inverse}

The finally step of the proof of the main theorem in this section is to prove that $\Psi^{-1}$ is continuous. We need the following lower bound estimate for the total energy of the (harmonic) Identity map of $\Sigma$ with respect to corresponding metrics.

\begin{lma}\label{l-proper} There are positive constants $C_1, C_2$ such that for any $\mathfrak{z}\in \mathcal{K}(\rho)$,
$$
E(\fz)\ge C_1\|\mathfrak{z}\|_\rho-C_2.
$$
where $E(\mathfrak{z})$ is the total energy of the identity map from $(\Sigma, \Psi(\mathfrak{z}))$ to $(\Sigma, \rho)$
\end{lma}
\begin{proof} Suppose on the contrary that this is not true. Then there exists a sequence $\fz_n$, with $a_n=\|\fz_n\|_\rho\to\infty$, such that $E(\fz_n)/a_n\to 0$. Let $u_n$ be the solutions to \eqref{e-Lambda-1} corresponding to $\fz_n$. Since $\mathcal{K}(\rho)$ is a finite dimensional inner product space, we may assume that $\fz_n/a_n\to \fz$ for some $\fz$ with $\|\fz\|_\rho=1$.  By Theorem \ref{t-u} and part (iv) of Proposition \ref{p-sup-est-1}, $v_n=u_n/a_n$ is a Cauchy sequence with respect to the supnorm. Let $v=\lim_{n\to\infty}u_n/a_n$. Since
$$
E(\fz_n)=\int_\Sigma u_n d\rho
$$
and $E(\fz_n)/a_n\to 0$, we conclude that $v\equiv0$.

Note that each $v_n$ satisfies:
\bee
\Delta_\rho v_n-v_n+\frac1{a_n}\lf(1+2|\xi_n|^2\ri)^\frac12=0
\eee
where $\xi_n$ as in \eqref{e-zeta-1} corresponding to $\fz_n$. Hence
\bee
\int_\Sigma\frac1{a_n}\lf(1+2|\xi_n|^2\ri)^\frac12 d\mu_\rho=\int_\Sigma v_n \mu_\rho\to0.
\eee
Multiplying the equation by $v_n$ and integrating by parts, we also have
\bee
\int_\Sigma\lf(|\nabla_\rho v_n|^2+v_n^2\ri)d\mu_\rho=\int_\Sigma\frac1{a_n}\lf(1+2|\xi_n|^2\ri)^\frac12 d\mu_\rho\to0.
\eee
On the other hand, we have
$$
\int_\Sigma(\Delta_\rho v_n-v_n)^2 d\mu_\rho=\frac1{a_n^2}\int_\Sigma\lf(1+2|\xi_n|^2\ri)d\mu_\rho
$$
obtained by simply integrating the square of both sides of the equation of $v_n$.
Now in an orthonormal frame of $\rho$, and write $\xi$ for $\xi_n$ etc, we have
\bee
\begin{split}
|\xi|^2=&|\fz|^2-\fz_{ab}(2u_{;ab}-\rho_{ab}\Delta_\rho u)+\frac14(2u_{;ab}-\rho_{ab}\Delta_\rho u)^2\\
=&|\fz|^2-2\fz_{ab}u_{ab}+|\nabla_\rho^2 u|^2-\frac12(\Delta_\rho u)^2
\end{split}
\eee
because $\fz$ is trace free with respect to $\rho$.
Hence we have
\begin{eqnarray*}
\lefteqn{\int_\Sigma\lf[(\Delta v_n)^2+2|\nabla v_n|^2 +v_n^2\ri]d\mu_\rho} \\
=& \frac{1}{a_n^2}A(\rho) +\int_\Sigma\lf(2\lf|\frac{\fz_n}{a_n}\ri|^2 +2|\nabla_\rho^2v_n|^2-(\Delta_\rho v_n)^2\ri)d\mu_\rho.
\end{eqnarray*}
Using the fact that Gaussian curvature of $\rho$ is $-\frac12$,
\bee
\begin{split}
\int_\Sigma 2|\nabla_\rho^2 v_n|^2d\mu_\rho=&-\int_\Sigma 2(v_n)_{a;ba}(v_n)_b d\mu_\rho\\
=&-\int_\Sigma \lf[ 2(v_n)_{a;ab}(v_n)_b-|\nabla_\rho v_n|^2\ri] d\mu_\rho\\
=&\int_\Sigma 2(\Delta_\rho v_n)^2 d\mu_\rho-\int_\Sigma |\nabla_\rho v_n|^2 d\mu_\rho
\end{split}
\eee
Hence we have
\bee
\int_\Sigma\lf(|\nabla_\rho v_n|^2+v_n^2\ri)d\mu_\rho=\frac{1}{a_n^2}A(\rho)+2\int_\Sigma \lf|\frac{\fz_n}{a_n}\ri|^2d\mu_\rho.
\eee
We conclude that
$$
\int_\Sigma \lf|\frac{\fz_n}{a_n}\ri|^2d\mu_\rho\to0
$$
as $n\to\infty$. This implies that $\fz=0$ which is impossible because $\|\fz\|_\rho=1$.
\end{proof}

\begin{remark}\label{r-proper} The lemma also follows from \cite[p.239--240]{Moncrief}. On the other hand, combining Corollary \ref{c-area} and the above lemma, we have
$$
C_1||\fz||-C_2\le E(\fz)\le C_2||\fz||+C_4
$$
for some positive constants $C_1,\dots,C_4$. This improves the estimates in \cite[p.240]{Moncrief}.

\end{remark}

\begin{lma}\label{l-inverse} $\Psi^{-1}$ is continuous.
\end{lma}
\begin{proof} Let $\mathfrak{z}_n\in \mathcal{K}(\rho)$ and $\gamma_n=\Phi(\mathfrak{z}_n)$ such that $\gamma_n\to \gamma$ in $C^\infty$ topology. In particular, the total energy of the identity map from $(\Sigma,\gamma_n)$ to $(\Sigma,\rho)$ is bounded above by a constant $C_1$ independent of $n$. By Lemma \ref{l-proper}, we have $\|\mathfrak{z}_n\|_\rho\le C_2$ for some constant $C_2$ independent of $n$. For any subsequence of $\mathfrak{z}_n$ we can find a subsequence which converges to some $\mathfrak{z}$. By the continuity of $\Psi$, we have $\Phi(\mathfrak{z})=\gamma$. By the fact that $\Psi$ is injective, we conclude that $\mathfrak{z}_n\to \Psi^{-1}(\gamma)$. This completes the proof of the lemma.

\end{proof}

Finally, we can prove the main theorem of this section:

\begin{proof}[Proof of Theorem \ref{t-homeomorphism}]
The theorem followings from Lemmas  \ref{l-oneone}, \ref{l-onto}, \ref{l-continuous}, and \ref{l-inverse}.  Here we have used the fact that if $\fz_i, \fz\in \mathcal{K}(\rho)$ so that $\fz_i\to \fz$ in $C^0$ norm, then $\fz_i\to \fz$ in $C^\infty$ norm, because $\fz_i, \fz$ can be expressed in terms of holomorphic functions in local coordinates.
\end{proof}

We would like to mention another application. Recall that one can define the energy of   $[h]\in \mathcal{T}$ as follows: Let $\gamma=s([h])$ so that $\mathbf{Id}: (\Sigma,\gamma)\to (\Sigma,\rho)$ is harmonic. Then the energy $E([h])$ of $[h]$ is defined as the total energy of the above map, see \cite{Tromba}. The following is \cite[Theorem 3.2.4]{Tromba}.
\begin{thm}\label{t-properness}
$E$ is a proper function.
\end{thm}
\begin{proof}
 The theorem is   a consequence of   Lemma \ref{l-proper} and Theorem \ref{t-homeomorphism}.
 \end{proof}


\begin{thebibliography}{1000}
\bibitem{AnderssonMoncrief}  Andersson, L.; Moncrief, V., {\sl
Elliptic-hyperbolic systems and the Einstein equations},
Ann. Henri Poincar\'e \textbf{ 4} (2003), no. 1, 1--34.
\bibitem{GilbargTrudinger} Gilbarg, D.;   Trudinger, N. S., {\sl Elliptic partial differential equations of second order}, second edition, Springer-Verlag, 1983.
\bibitem{Moncrief} Moncrief, V., {\sl Relativistic Teichm\"{u}ller Theory -- A Hamilton-Jacobi Approach to $2+1$-Dimensional Einstein Gravity}, Surveys in Differential Geometry XII, International Press, 2018, pp. 203-249
\bibitem{Tromba} Tromba, A. J., {\sl Teichm\"uller theory in Riemannian geometry},  Birkh\"auser Verlag, 1992.

\bibitem{Witten} Witten, E., {\sl
2+1
-dimensional gravity as an exactly soluble system}, Nuclear Phys. B \textbf{311} (1988/89), no. 1, 46--78.

\bibitem{Wolf} Wolf, M., {\sl The Teichm\"uller theory of harmonic maps}, J. Differential Geom. \textbf{29} (1989), no. 2, 449--479.

\end{thebibliography}
\end{document}